\documentclass[12pt]{amsart}

\usepackage[dvipdfmx]{graphicx}
\usepackage[dvipdfmx]{color}
\usepackage{cite}
\usepackage{amsthm, amsmath, amssymb, mathtools}
\usepackage[top=45truemm, bottom=45truemm, left=30truemm, right=30truemm]{geometry}

\renewcommand{\epsilon}{\varepsilon}
\renewcommand{\limsup}{\varlimsup}
\renewcommand{\liminf}{\varliminf}

\numberwithin{equation}{section}
\newtheorem{theorem}{Theorem}[section]
\newtheorem{lemma}[theorem]{Lemma}
\newtheorem{corollary}[theorem]{Corollary}

\newtheorem{proposition}[theorem]{Proposition}

\newtheorem{question}[theorem]{Question}

\theoremstyle{definition}
\newtheorem{remark}[theorem]{Remark}
\newtheorem{notation}[theorem]{Notation}

\newcommand{\diam}{\mathrm{diam}}
\newcommand{\Haus}{\mathrm{dim}_{\mathrm{H}}\:}
\newcommand{\cH}{\mathcal{H}}
\newcommand{\PS}{\mathrm{PS}}

\title[Linear Diophantine equations in Piatetski-Shapiro sequences]{Linear Diophantine equations in \\ Piatetski-Shapiro sequences}
\author[T.\ Matsusaka]{Toshiki Matsusaka}
\address{Toshiki Matsusaka \\
Institute for Advanced Research\\ Nagoya University\\ Furo-cho\\ Chikusa-ku\\ Nagoya\\ 464-8602\\ Japan}
\curraddr{}
\email{matsusaka.toshiki@math.nagoya-u.ac.jp}

\author[K.\ Saito]{Kota Saito}
\address{Kota Saito\\
Graduate School of Mathematics\\ Nagoya University\\ Furo-cho\\ Chikusa-ku\\ Nagoya\\ 464-8602\\ Japan}
\curraddr{}
\email{m17013b@math.nagoya-u.ac.jp}

\subjclass[2010]{Primary:11K55, Secondary:11D04}

\keywords{Piatetski-Shapiro sequence, Hausdorff dimension, Diophantine equation, arithmetic progression, discrepancy}
\date{}

\begin{document}
\maketitle
\begin{abstract}
A Piatetski-Shapiro sequence with exponent $\alpha$ is a sequence of integer parts of $n^\alpha$ $(n = 1,2,\ldots)$ with a non-integral $\alpha > 0$. We let $\PS(\alpha)$ denote the set of those terms. In this article, we study the set of $\alpha$ so that the equation $ax + by = cz$ has infinitely many pairwise distinct solutions $(x,y,z) \in \PS(\alpha)^3$, and give a lower bound for its Hausdorff dimension. As a corollary, we find uncountably many $\alpha > 2$ such that $\PS(\alpha)$ contains infinitely many arithmetic progressions of length $3$. 
\end{abstract}

\section{Introduction}
Let $\lfloor x \rfloor$ denote the integer part of $x\in \mathbb{R}$. For a non-integral $\alpha>0$, the sequence $(\lfloor n^\alpha\rfloor)_{n=1}^\infty$ is called the Piatetski-Shapiro sequence with exponent $\alpha$ and let $\PS(\alpha)=\{\lfloor n^\alpha \rfloor \colon n\in \mathbb{N}\}$. We say that an equation 
$
f(x_1, \ldots , x_n)=0
$
is \textit{solvable} in $\PS(\alpha)$ if there are infinitely many pairwise distinct tuples $(x_1,\ldots,x_n)\in \PS(\alpha)^n$ satisfying this equation. In this article, we investigate the solvability in $\PS(\alpha)$ of linear Diophantine equations 
\begin{equation}\label{Equation-ax+by=cz}
ax+by=cz
\end{equation}
for all fixed $a,b,c \in \mathbb{N}$. For example, the solvability of the equation $y = \theta x + \eta$ for $\theta, \eta \in \mathbb{R}$ with $\theta \not\in \{0,1\}$ has been studied by Glasscock \cite{Glasscock17, Glasscock20}. He asserts that if the equation $y = \theta x + \eta$ has infinitely many solutions $(x,y) \in \mathbb{N}^2$, then for Lebesgue-a.e. $\alpha > 1$ it is solvable or not in $\PS(\alpha)$ according as $\alpha < 2$ or $\alpha > 2$. As a direct consequence, for Lebesgue-a.e. $1 < \alpha < 2$, the equation $z= (a/c)x + (b/c)$ is solvable in $\PS(\alpha)$ for all $a,b,c \in \mathbb{N}$ with $\gcd(a,c) | b$. In other words, the equation \eqref{Equation-ax+by=cz} with $\gcd(a,c)|b$ is solvable in $\PS(\alpha)$. On the other hand, for $\alpha > 2$, we did not know at all whether the equation \eqref{Equation-ax+by=cz} is solvable in $\PS(\alpha)$ or not.

Our main result provides an answer to this question. We consider the set of $\alpha$ in a short interval $[s,t] \subset (2, \infty)$ so that \eqref{Equation-ax+by=cz} is solvable. Then the following theorem asserts that its Hausdorff dimension is positive. To state the theorem, let $\{ x \}$ be the fractional part of $x \in \mathbb{R}$, and $\Haus X$ the Hausdorff dimension of $X \subseteq \mathbb{R}$, of which definition will be given in Section~\ref{Section-preparations}.

\begin{theorem}\label{Theorem-fractal_dimension}
Let $a,b,c\in \mathbb{N}$. For all positive real numbers $2<s<t$, we have
\begin{align*}
&\dim_\mathrm{H}(\{\alpha \in [s,t] \colon \text{$ax+by=cz$ is solvable in $\PS(\alpha)$} \}) \\
&\geq 
\begin{cases}
\displaystyle  \left(s + \frac{s^3}{(2+\{s\}-2^{1-\lfloor s\rfloor} )(2-\{s\}) } \right)^{-1}  & \text{if $a=b=c$}\\[10pt]
\displaystyle 2\left(s+\frac{s^3}{(2+\{s\}-2^{1-\lfloor s\rfloor} )(2-\{s\}) } \right)^{-1} & \text{otherwise}.
\end{cases}
\end{align*}
\end{theorem}

The positiveness of the Hausdorff dimension implies that this set is uncountable for any closed interval $[s, t] \subset (2, \infty)$. Moreover, we can easily conclude the following.

\begin{corollary}\label{Corollary-topological_result}
For any closed interval $I\subset (2,\infty)$, the set of $\alpha \in I$ such that $ax+by=cz$ is solvable in $\PS(\alpha)$ is uncountable and dense in $I$.
\end{corollary}

In particular, for $a=b=1, c=2$, a pairwise distinct tuple $(x,z,y)$ satisfying \eqref{Equation-ax+by=cz} forms an arithmetic progression of length $3$. Therefore Corollary~\ref{Corollary-topological_result} implies 

\begin{corollary}\label{Corollary-APs} 
For any closed interval $I\subset (2,\infty)$, the set of $\alpha \in I$ such that $\PS(\alpha)$ contains infinitely many arithmetic progressions of length $3$ is uncountable and dense in $I$.
\end{corollary}

 There are some related works on arithmetic progressions and Piatetski-Shapiro sequences. Frantzikinakis and Wierdl \cite{FrantzikinakisWierdl} showed that $\PS(\alpha)$ contains arbitrarily long arithmetic progressions for all $1<\alpha<2$. They also considered any set of positive integers with positive upper density, and proved that every such set contains arbitrarily long arithmetic progressions whose gap difference belongs to $\PS(\alpha)$ for all non-integral $\alpha>1$. Here we say that $A\subseteq \mathbb{N}$ has positive upper density if $\limsup_{N\rightarrow \infty}|A\cap \{1,\ldots,N\}|/N>0 $.  This result is one of extensions of Szemer\'edi's theorem \cite{Szemeredi}. Furthermore, the second author and Yoshida \cite{SaitoYoshida} gave another extension of Szemer\'edi's theorem to Piatestki-Shapiro sequences. They showed that for any $A\subseteq \mathbb{N}$ with positive upper density, the set $\{\lfloor n^\alpha \rfloor \colon n\in A\}$ with $1 < \alpha < 2$ contains arbitrarily long arithmetic progressions, and posed a question.
 
\begin{question}[{\cite[Question~13]{SaitoYoshida}}]\label{Question-SY}
Is it ture that
\[
\sup\{ \alpha \ge 1 \colon \text{$ \PS(\alpha)$ contains arbitrarily long arithmetic progressions}\}=2?
\]
\end{question}
We do not get any answer to this question here, but surprisingly by Corollary~\ref{Corollary-APs}, the supremum of $\alpha$ such that $\PS(\alpha)$ contains infinitely many arithmetic progressions of length $3$ is positive infinity. Glasscock also posed a related question to the equation \eqref{Equation-ax+by=cz} for $a=b=c=1$.
\begin{question}[{\cite[Question~6]{Glasscock17} }]\label{Question-Glasscock}
Does there exist an $\alpha_S>1$ with the property that for Lebesgue-a.e. or all $\alpha>1$, the equation $x + y = z$ is solvable or not in $\PS(\alpha)$ according as $\alpha < \alpha_S$ or $\alpha > \alpha_S$?
\end{question} 
By Corollary~\ref{Corollary-topological_result}, the case with ``all $\alpha>1$'' in Question~\ref{Question-Glasscock} is false since the supremum of $\alpha>0$ such that \eqref{Equation-ax+by=cz} is solvable in $\PS(\alpha)$ is positive infinity.  However, the case with ``Lebesgue-a.e.'' in Question~\ref{Question-Glasscock} is still open.

The rest of the article is organized as follows. First in Section \ref{Section-preparations} we define the discrepancy of the sequences and the Hausdorff dimension, and describe some known useful results. In Section \ref{Section-proofs_of_lemmas} and \ref{Section-dimension}, we prove a series lemmas. Finally we provide a proof of Theorem~\ref{Theorem-fractal_dimension}.

\begin{notation}
Let $\mathbb{N}$ be the set of all positive integers, $\mathbb{Z}$ the set of all integers, $\mathbb{Q}$ the set of all rational numbers, and $\mathbb{R}$ the set of all real numbers. For $x\in \mathbb{R}$, let $\lfloor x\rfloor $ denote the integer part of $x$, $\{x\}$ denote the fractional part of $x$, and $\lceil x \rceil$ denote the minimum integer $n$ such that $x\leq n$.  A tuple $(x_1,\ldots, x_k)\in \mathbb{R}^k$ is called pairwise distinct if $\#\{x_1,\ldots , x_k\} =k$. Let $\sqrt{-1}$ denote the imaginary unit, and define $e(x)$ by $e^{2\pi \sqrt{-1}x} $ for all $x\in \mathbb{R}$.
 \end{notation}

\section{Preparations}\label{Section-preparations}

For all $\mathbf{x}=(x_1,x_2,\ldots, x_d)\in \mathbb{R}^d$, define $\{\mathbf{x}\} = (\{x_1\}, \{x_2\},\ldots, \{x_d\})$.
Let $(\mathbf{x}_n)_{1\leq n\leq N}$ be a sequence composed of $ \mathbf{x}_n\in \mathbb{R}^d$ for all $1\leq n\leq N$. We define the \textit{discrepancy} of the sequence $(\mathbf{x}_n)_{1\leq n\leq N}$ by  
\[
D(\mathbf{x}_1,\ldots, \mathbf{x}_N)=\sup_{\substack{0\leq a_i<b_i\leq 1 \\ {1\leq i\leq d}}} \left| \frac{\#\left\{ n\in \mathbb{N}\cap[1,N] \colon \{\mathbf{x}_n\} \in \prod _{i=1}^d [a_i,b_i) \right\}}{N}- \prod_{i=1}^d (b_i-a_i)  \right|.
\]
In order to evaluate an upper bound for the discrepancy,  we use the following inequality which is shown by Koksma \cite{Koksma} and Sz\"{u}sz \cite{Szusz} independently: there exists $C_d>0$ which depends only on $d$ such that for all $K\in \mathbb{N}$, we have
\begin{equation}\label{Inequality-Koksma_Szusz}
D(\mathbf{x}_1,\ldots, \mathbf{x}_N)
\leq C_d\left( \frac{1}{K}
+\sum_{\substack{0<\|\mathbf{k}\|_\infty \leq K \\ \mathbf{k}\in \mathbb{Z}^d}}\frac{1}{\nu(\mathbf{k})} \left|\frac{1}{N} \sum_{n=1}^N e^{2\pi \sqrt{-1} \langle \mathbf{k},\mathbf{x}_n\rangle} \right| \right), 
\end{equation}
where we let $\langle \cdot ,\cdot\rangle$ denote the standard inner product and define
\[
\|\mathbf{k}\|_\infty=\max\{|k_1|,\ldots, |k_d|\},\quad \nu(\mathbf{k})=\prod_{i=1}^d \max\{1,|k_i|\}.
\]
This inequality is sometimes reffered as the Erd\H{o}s-Tur\'an-Koksma inequality. We refer \cite{DrmotaTichy} to the readers for more details on discrepancies and a proof of \eqref{Inequality-Koksma_Szusz}. This inequality reduces the estimate of the discrepancy to that of exponential sums. Furthermore, the exponential sum is evaluated by the following lemma.

\begin{lemma}[van der Corput's $k$-th derivative test]\label{Lemma-van_der_Corput}
Let $f(x)$ be real and have continuous derivative up to $k$-th order, where $k\geq 4$. Let $\lambda_k \leq f^{(k)} (x) \leq h \lambda_k$ $($or the same as for $-f^{(k)}(x))$. Let $b-a\geq 1$. Then there exists $C(h,k)>0$ such that 
\[
\left|\sum_{a<n\leq b} e^{2\pi \sqrt{-1} f(n)} \right| \leq C(h,k) \left((b-a)\lambda_k^{1/(2^k-2)}+(b-a)^{1-2^{2-k}}\lambda_k^{-1/(2^k-2)} \right).
\] 
\end{lemma}
\begin{proof}
See Titchmarsh's book \cite[Theorem~5.13]{Titchmarsh}.
\end{proof}

We next introduce the Hausdorff dimension. For every $U\subseteq \mathbb{R}$, write the diameter of $U$ by $\diam(U)=\sup_{x,y\in U}|x-y|$. Fix $\delta>0$. For all $F\subseteq \mathbb{R}$ and $s\in [0,1]$, we define
\[
\cH_{\delta}^s (F)=\inf \left\{\sum_{j=1}^\infty \diam(U_j)^s\colon F\subseteq \bigcup_{j=1}^\infty U_j,\ (\forall j\in \mathbb{N})\ \diam(U_j)\leq \delta   \right\},
\]
and $\cH^s(F)=\lim_{\delta\rightarrow +0} \cH_\delta^s(F)$ is called the $s$-\textit{dimensional Hausdorff measure} of $F$. Further,
\[
\Haus F=\inf \{s\in [0,1] \colon \mathcal{H}^s(F)=0 \}
\]
is called the \textit{Hausdorff dimension} of $F$. Note that the Hausdorff dimension can be defined on all metric spaces, but we use only $\mathbb{R}$ in this article. By the definition, the following basic properties hold: 
\begin{itemize}
\item (Monotonicity) for all $F\subseteq E \subseteq \mathbb{R}$, we have $\Haus F \leq \Haus E$;
\item (Countable stability) if $F_1, F_2, \ldots \subseteq \mathbb{R} $ is a countable sequence of sets, then we have $\Haus \bigcup_{n=1}^\infty F_n =\sup_{n\in \mathbb{N}} \Haus F_n$.   
\end{itemize}

We refer Falconer's book \cite{Falconer} for the readers who want to know more details on fractal dimensions.
In order to prove Theorem~\ref{Theorem-fractal_dimension}, we construct a general Cantor set which is a subset of the set of all $\alpha$ such that \eqref{Equation-ax+by=cz} is solvable in $\PS(\alpha)$. In Falconer's book \cite[(4.3)]{Falconer}, we can see a general construction of Cantor sets and a technique to evaluate the Hausdorff dimension of them as follows: Let $[0,1]=E_0\supseteq E_1 \supseteq E_2 \cdots$ be a decreasing sequence of sets, with each $E_k$ a union of a finite number of disjoint closed intervals called $k$-th \textit{level basic intervals}, with each interval of $E_k$ containing at least two intervals of $E_{k+1}$, and the maximum length of $k$-th level intervals tending to $0$ as $k\rightarrow \infty$. Then let
\begin{equation}\label{Equation-Cantor_set}
F= \bigcap_{k=0}^\infty E_k.
\end{equation} 

\begin{lemma}[{\cite[Example~4.6 (a)]{Falconer}}]\label{Lemma-lower_bounds_for_dimension_origin}
Suppose in the general construction \eqref{Equation-Cantor_set} each $(k-1)$-st level interval contains at least $m_k\geq 2$ $k$-th level intervals $(k=1,2,\ldots)$ which are separated by gaps of at least $\delta_k$, where $0<\delta_{k+1}<\delta_k$ for each $k$. Then  
\[
\Haus F \geq \liminf_{k\rightarrow \infty} \frac{\log m_1\cdots m_{k-1}}{-\log(m_k\delta_k)}.
\]
\end{lemma}

Since the Hausdorff dimension is stable under similarity transformations, the initial interval $E_0$ may be taken arbitrarily closed set. Moreover, let $E'_k$ be the set of inner points of $E_k$ for all $k \in \mathbb{N}$. Then the Hausdorff dimension of $\bigcap_{k =0}^\infty E'_k$ is equal to that of $\bigcap_{k =0}^\infty E_k$. In fact, let $\mathcal{N}_k$ be the boundary of $E_k$, that is, the set of all end points of $k$-th level intervals. We easily see that
\[
	\mathcal{N} := F \setminus \left(\bigcap_{k=0}^\infty E'_k \right) \subset \bigcup_{k=0}^\infty \mathcal{N}_k =: \mathcal{N}_\infty.
\]
Since each $\mathcal{N}_k$ is a finite set, the set $\mathcal{N}_\infty$ is a countable set. By the monotonicity, and the fact that the Hausdorff dimension of a countable set is $0$, we get
\[
	0 \leq \Haus \mathcal{N} \leq \Haus \mathcal{N}_\infty = 0,
\]
that is, $\Haus \mathcal{N} = 0$. Therefore by the countable stability, we have
\begin{align*}
	\Haus F = \max \left\{ \Haus \bigg(\bigcap_{k=0}^\infty E'_k\bigg), \Haus \mathcal{N} \right\} = \Haus \bigg(\bigcap_{k=0}^\infty E'_k\bigg).
\end{align*}
By summarizing this discussion, we have the following:

 \begin{lemma}\label{Lemma-lower_bounds_for_dimension}
Let $E_0$ be any open interval, and let $E_0\supseteq E_1 \supseteq E_2 \cdots$ be a decreasing sequence of sets, with each $E_k$ a union of a finite number of disjoint open intervals, and the maximum length of $k$-th level intervals tending to $0$ as $k\rightarrow \infty$. Suppose each $(k-1)$-st level interval contains at least $m_k\geq 2$ $k$-th level intervals $(k=1,2,\ldots)$ which are separated by gaps of at least $\delta_k$, where $0<\delta_{k+1}<\delta_k$ for each $k$. Then  
\[
\Haus \bigcap_{k=0}^\infty E_k  \geq \liminf_{k\rightarrow \infty} \frac{\log m_1\cdots m_{k-1}}{-\log(m_k\delta_k)}.
\]
\end{lemma}

\section{Lemmas I}\label{Section-proofs_of_lemmas}
We write $O(1)$ for a bounded quantity. If this bound depends only on some parameters $a_1,\ldots, a_n$, then for instance we write $O_{a_1,a_2,\ldots, a_n}(1)$. As is customary, we often abbreviate $O(1)X$ and $O_{a_1,\ldots, a_n}(1)X$ to $O(X)$ and $O_{a_1,\ldots, a_n}(X)$ respectively for a non-negative quantity $X$. We also say $f(X) \ll g(X)$ and $f(X) \ll_{a_1,\ldots, a_n} g(X)$  as $f(X)=O(g(X))$ and $f(X)=O_{a_1,\ldots, a_n}(g(X))$ respectively, where $g(X)$ is non-negative. 

Let us consider the solvability of the equation \eqref{Equation-ax+by=cz}. In this and next sections, we fix $a,b,c,d\in \mathbb{N}$ with $d\geq 2$ and $\beta,\gamma\in \mathbb{R}$ with $d<\beta<\gamma<d+1$. Unless it causes confusion, we do not indicate their dependence hereinafter. Take a large parameter $x_0=x_0(a,b,c,d,\beta,\gamma)>0$. For all integer $x\geq x_0$, we define
\begin{equation*}
J_{a,b,c}(x)=\left\{ 
\begin{array}{ll}
\displaystyle  \left(\left(\frac{b}{cx^2\log x}+\frac{a}{c} \right)^{1/\gamma} x,\ \left(a/c\right)^{1/\beta}x \right)_{\mathbb{N}} \setminus \{nx \colon n\in \mathbb{N} \} & \text{ if }c<a,\\
\\
\displaystyle \left(\left(\frac{a}{c-b(x^2\log x)^{-1}}\right)^{1/\beta}x,\ \left(a/c\right)^{1/\gamma} x \right)_\mathbb{N} & \text{ if }a< c,\\
\\
\displaystyle \left(2^{1/\gamma}\left(x+\frac{1}{x\lceil \log x\rceil}\right),\ 2^{1/\beta}x \right)_\mathbb{N} & \text{ if }a=b=c,  \\ 
\end{array}
\right.
\end{equation*}
where let $(s,t)_{\mathbb{N}}$ denote $(s,t)\cap \mathbb{N}$. Note that $J_{a,b,c}(x)$ is non-empty if $x_0$ is sufficiently large. In the case when $a=c$ and $b\neq c$, $J_{a,b,c}(x)$ is not defined above, however this case comes down to the case when $a\neq c$ by switching the roles of $(a,x)$ and $(b,y)$. Thus the three cases in the definition of $J_{a,b,c}(x)$ are essential.

\begin{lemma}\label{Lemma-alpha(x,z)_1_2}
Assume that $a\neq c$. Then there exists $C>0$ such that for all integer $x\geq x_0$ and for all $z \in J_{a,b,c}(x)$, we can find $\alpha=\alpha(x,z)\in (\beta, \gamma)$ so that $ax^{\alpha}+b=cz^{\alpha}$, and
\begin{equation}\label{Equation-asymptotic_for_alpha_1_2}
\left|\alpha-\frac{\log(a/c)}{\log(z/x)}\right| \leq \frac{C}{x^2\log x}.
\end{equation}
\end{lemma}

\begin{proof}
Fix any $x\geq x_0$ and $z \in J_{a,b,c}(x)$. For all $u\in \mathbb{R}$, define a continuous function $f(u)=ax^u+b-cz^u$.
We prove the claim by considering two cases $a>c$ and $c>a$. \\

\textbf{Step~1.} In the case $a>c$, let
\[
\alpha_0=\frac{\log(a/c)}{\log(z/x)},\quad \alpha_1=\frac{\log(a/c+b/(cx^2\log x)) }{\log(z/x)}.
\]
Then, $z\in J_{a,b,c}(x)$ implies $\beta< \alpha_0<\alpha_1< \gamma$. It follows that $f(\alpha_0)= b>0$.  If necessary, by taking a larger $x_0$, we have
\[
f(\alpha_{1}) = x^{\alpha_1} (a+bx^{-\alpha_1}-c(z/x)^{\alpha_1})\leq x^{\alpha_1}(a+b/(x^2 \log x)- c(z/x)^{\alpha_1} )= 0.
\]
Therefore, by the intermediate value theorem, there exists a zero $\alpha = \alpha(x,z)$ of $f$ such that $\beta<\alpha_0\leq \alpha\leq \alpha_1< \gamma$. We obtain \eqref{Equation-asymptotic_for_alpha_1_2} since  
\[
|\alpha_1-\alpha_0|=\frac{ \log (1+ b/(ax^2\log x))}{\log (z/x)}\ll_{a,b,c} \frac{\log (a/c)}{\log (z/x)}\cdot \frac{1}{x^2\log x} \leq \frac{\gamma}{x^2\log x}.
\]

\textbf{Step~2.} In the case $c>a$, let
\[
\alpha_0=\frac{\log(c/a)}{\log(x/z)},\quad \alpha_{1}'=\frac{\log(c/a-b/(ax^2\log x))}{\log (x/z)}.
\]
By $z\in J_{a,b,c}(x)$, $\beta< \alpha_{1}'  <\alpha_0< \gamma$ and $x\ll_{a,b,c,\beta,\gamma} z$ hold. Then by the calculation in Step~1, $f(\alpha_0)\geq 0$. Further $x\ll z$ implies $z^{-\alpha_{1}'}\leq z^{-\beta}\ll x^{-\beta}$. Thus if $x_0$ is sufficiently large, we have $z^{-\alpha_{1}'}\leq 1/(x^2\log x)$, which yields that
\[
f(\alpha_{1}')\leq z^{\alpha_{1}'} (a(x/z)^{\alpha_{1}'} +bz^{-\alpha_{1}'} -c)\leq  z^{\alpha_{1}'} (a(x/z)^{\alpha_{1}'} +b/(x^2\log x) -c)= 0.
\]
Therefore, by the intermediate value theorem, there exists a zero $\alpha=\alpha(x,z)$ of $f$ such that $\beta< \alpha_{1}'\leq \alpha \leq \alpha_0< \gamma$. We obtain \eqref{Equation-asymptotic_for_alpha_1_2} since
\[
|\alpha_0-\alpha_{1}'|=\frac{ |\log (1- b/(cx^2\log x))|}{\log (x/z)}\ll_{a,b,c} \frac{\log (c/a)}{\log (x/z)}\cdot \frac{1}{x^2\log x} \leq \frac{\gamma}{x^2\log x}.
\]
\end{proof}

\begin{lemma}\label{Lemma-alpha(x,z)}
There exists $C>0$ such that for all integer $x\geq x_0$ and $z\in J_{1,1,1}(x)$, we can find
$\alpha=\alpha(x,z)\in( \beta, \gamma)$ so that
$x^\alpha +\left(x+ (x\lceil \log x\rceil )^{-1}  \right)^\alpha =z^\alpha$,
and
\begin{equation}\label{Equation-asymptotic_formula_for_alpha}
\left|\alpha- \frac{\log 2}{\log (z/x)}\right| \leq \frac{C}{x^2\log x }.
\end{equation}
\end{lemma}

\begin{proof}
Take any $x\geq x_0$ and $z\in J_{1,1,1}(x)$. For all $u\in \mathbb{R}$, define a continuous function
$f(u) =x^u +(x+(x\lceil \log x\rceil )^{-1} )^u -z^u$, and set
\[
\alpha_0= \frac{\log 2}{\log (z/x)}  ,\quad \alpha_1=\frac{\log 2}{\log\left(\cfrac{z}{x+ (x\lceil \log x\rceil)^{-1}}\right)}.
\]
By $z\in J_{1,1,1}(x)$, $\beta< \alpha_0<\alpha_1< \gamma$ holds. By the definitions of $\alpha_0$ and $\alpha_1$, we have
\[
f(\alpha_0) >z^{\alpha_0} \left(\frac{1}{2}+\frac{1}{2}-1\right)=0,\quad
f(\alpha_1) <z^{\alpha_1} \left(\frac{1}{2}+\frac{1}{2}-1  \right)=0.  
\]
Therefore, by the intermediate value theorem, there exists a zero $\alpha=\alpha(x,z)$ of $f$ such that $\alpha_0\leq \alpha \leq \alpha_1$. Further, we conclude \eqref{Equation-asymptotic_formula_for_alpha} since  
\[
|\alpha_1-\alpha_0|\leq \frac{\gamma^2}{\log 2} \log\left(1+\frac{1}{x^2\log x }\right)\leq \frac{\gamma^2}{\log 2}\cdot \frac{1}{x^2\log x }.
\]
\end{proof}

\begin{lemma}\label{Lemma-n_0}
Let $\epsilon>0$ be an arbitrarily small real number. For all $X,Y,Z\in \mathbb{N}$, and $\alpha\in \mathbb{R}$ with $\beta<\alpha<\gamma$, if we have 
\begin{equation}\label{Equation-XYZalpha}
aX^\alpha+bY^{\alpha}=cZ^{\alpha},
\end{equation}
then there exists $n_0\in \mathbb{N}$ such that 
\begin{gather}\label{Equation-n_0_and_fractional_part}
a\lfloor (n_0X)^\alpha \rfloor +b\lfloor(n_0Y)^\alpha\rfloor =c\lfloor (n_0Z)^\alpha \rfloor, \\ \label{Inequality-fractional_parts}
\max(\{(n_0 X)^\alpha\}, \{(n_0Y)^\alpha\}, \{(n_0Z)^\alpha\} )<\frac{1}{2},\\ \label{Inequality-n_0}
n_0 \ll_{\epsilon}  (X+Y)^{\gamma^2/((2+\{\beta\}-2^{1-\lfloor \beta\rfloor} )(2-\{\gamma\})) +\epsilon }. 
\end{gather}
\end{lemma}

\begin{proof} 
Choose $X,Y,Z\in \mathbb{N}$ and $\alpha$ with $\beta<\alpha<\gamma$ satisfying \eqref{Equation-XYZalpha}. For all $n\in \mathbb{N}$,
\begin{align*}
c\lfloor (nZ)^\alpha\rfloor=c(nZ)^\alpha -c\{(nZ)^\alpha\} =a\lfloor (nX)^\alpha \rfloor +b\lfloor (nY)^\alpha \rfloor+\delta(n),
\end{align*}
where define $\delta(n)=a\{(nX)^\alpha\} +b\{(nY)^\alpha\}- c\{(nZ)^\alpha\}$. Let
\[
A=\left\{n\in \mathbb{N} \colon |\delta(n)|<1,\ \max(\{(n X)^\alpha\}, \{(nY)^\alpha\}, \{(nZ)^\alpha\} )<\frac{1}{2} \right\},
\]
then any $n\in A$ satisfies \eqref{Equation-n_0_and_fractional_part} and \eqref{Inequality-fractional_parts}. Let us show the existence of $n\in A$ satisfying \eqref{Inequality-n_0}. Take a sufficiently large parameter $R=R(a,b,c,d,\beta,\gamma,\epsilon)$, and put
\begin{equation}\label{Equation-N}
N=\left\lceil R  (X+Y)^{\gamma^2/((2+\{\beta\}-2^{1-\lfloor \beta\rfloor} )(2-\{\gamma\})) +\epsilon} \right\rceil.
\end{equation}
Take a small $\xi = \xi(d, \beta, \gamma, \varepsilon) > 0$, and put $\psi = \{\beta\} - 2 + (2^{d+2}-2)(1/2^d - 2\xi)$. Since this is reformulated to
\begin{equation}\label{Equation-psi}
\psi=2+\{\beta\}-2^{1-\lfloor\beta \rfloor}+O(\xi),
\end{equation}
we have $0<\psi<\alpha$ for a small enough $\xi$. Moreover, we let $L(h_1, h_2) = (h_1 X^\alpha + h_2 Y^\alpha)/c$.\\

\textbf{Step~1.} We firstly discuss the case when 
\begin{equation}\label{Inequality-assumption_on_L}
|L(h_1,h_2)| \geq N^{-\psi}
\end{equation}
holds for all $h_1,h_2 \in \mathbb{Z}$ with $0<\max\{|h_1|,|h_2|\}\leq  N^\xi$. In this case, define
\begin{equation}\label{Inequality-(nx)alpha}
A_1=\left\{n\in \mathbb{N} \colon 0\leq  \{(nX)^\alpha/c\} < \frac{1}{4ac},\quad 0\leq  \{(nY)^\alpha/c\} < \frac{1}{4bc}\right\}.
\end{equation}
Then we have $A_1\subseteq A$. In fact, take any $n\in A_1$. We see that
\begin{equation}\label{Equation-(nX)}
(nX)^\alpha=c\lfloor (nX)^\alpha /c\rfloor +c\{(nX)^\alpha/c \}.
\end{equation}
Since the first term on the right-hand side of \eqref{Equation-(nX)} is an integer and the second term belongs to $[0,1)$ by $n\in A_1$, we get $\{(nX)^\alpha\}=c\{(nX)^\alpha/c\} $. This yields that $\{(nX)^\alpha\} < 1/(4a)$. Similarly, $\{(nY)^\alpha\} <1/(4b)$ holds. Further, 
\[
\{(nZ)^\alpha \}=\{ a(nX)^\alpha /c +b(nY)^\alpha/c\} \leq a\{ (nX)^\alpha /c\}+b\{(nY)^\alpha/c\}< \frac{1}{2c}.
\]
Hence we have
\[
|\delta(n)|\leq a\{(nX)^\alpha\}+b\{(nY)^\alpha\} +c\{(nZ)^\alpha\}<  \frac{1}{4}+\frac{1}{4}+\frac{1}{2}=1.
\]
Therefore $A_1\subseteq A$ holds.

We now evaluate the distribution of $A_1$. Let $D_1(N)$ be the discrepancy of the sequence $((nX)^\alpha/c,\ (nY)^\alpha/c)_{N < n \leq 2N}$. The inequality \eqref{Inequality-Koksma_Szusz} with $K= \lfloor N^{\xi} \rfloor$ implies that 
\[
D_1(N)\ll N^{-\xi} +\sum_{0<\|(h_1, h_2)\|_\infty \leq N^\xi} \frac{1}{\nu (h_1,h_2)} \left|\frac{1}{N}  \sum_{N<n\leq 2N} e(L(h_1,h_2)n^\alpha)  \right|. 
\]
For all $u\in \mathbb{R}$, define $f(u)=L(h_1,h_2)u^\alpha$. For each $N<u\leq 2N$,
\[
|L(h_1, h_2)| N^{\alpha -(d+2)} \ll  |f^{(d+2)} (u) |\ll | L(h_1, h_2)| N^{\alpha -(d+2)}.
\]
Therefore Lemma~\ref{Lemma-van_der_Corput} with $k=d+2$ yields that
\begin{align*}
&\frac{1}{N}  \sum_{N<n\leq 2N} e(L(h_1, h_2)n^\alpha) \\
&\ll  ( |L(h_1, h_2)| N^{\alpha -(d+2)})^{1/(2^{d+2}-2)} +\frac{( |L(h_1, h_2)| N^{\alpha -(d+2)})^{-1/(2^{d+2}-2)}}{N^{1/2^d}} \\
&\ll  ( L(N^\xi, N^\xi) N^{\{\gamma\}-2 })^{1/(2^{d+2}-2)} +\frac{ N^{(2-\{\beta\}+\psi)/(2^{d+2}-2) }}{N^{1/2^d}}.
\end{align*}
By the definition of $\psi$, it follows that $(2-\{\beta\}+\psi)/(2^{d+2}-2)-1/2^d=-2\xi$. Then  
\begin{align*}
\frac{1}{N}  \sum_{N<n\leq 2N} e(L(h_1, h_2)n^\alpha) \ll   \left( (X+Y)^\gamma N^{\{\gamma\}-2+\xi} \right)^{1/(2^{d+2}-2)}
+N^{-2\xi}.
\end{align*}
Therefore, since $\log N\ll_t N^t$ for all $t>0$, we have  
\begin{equation}\label{Inequality-D_1(N)}
D_1(N)\ll_{\xi} N^{-\xi} +  \left( (X+Y)^\gamma N^{\{\gamma\}-2+2\xi} \right)^{1/(2^{d+2}-2)}.
\end{equation}
Let $E_1(N)$ be the right-hand side of \eqref{Inequality-D_1(N)}. By the definition of the discrepancy, 
\[
\frac{\# (A_1\cap (N,2N])}{N}= \frac{1}{16abc^2} +O_{\xi}\left(E_1(N)\right). 
\]
By \eqref{Equation-N}, we have
\begin{equation}\label{Inequality-(X+Y)}
(X+Y)^\gamma N^{\{\gamma\}-2+2\xi } \ll R^{\{\gamma\}-2+2\xi } (X+Y)^{e}. 
\end{equation}
Here the exponent $e$ of $(X+Y)$ on the right-hand side of \eqref{Inequality-(X+Y)} is negative since
\begin{align*}
 e &=\gamma+(\{\gamma\}-2+2\xi)\left(\frac{\gamma^2}{(2+\{\beta\}-2^{1-\lfloor\beta\rfloor})(2-\{\gamma\}) } +\epsilon\right)\\
	&= \gamma \left(1 - \frac{\gamma}{2+\{\beta\} - 2^{1-\lfloor \beta\rfloor}} \right) - \varepsilon(2-\{\gamma\}) + O(\xi)\\
	&\leq \gamma\cdot \frac{2+\{\beta\}-\gamma }{2+\{\beta\}-2^{1-\lfloor \beta\rfloor }} -\epsilon(2-\{\gamma\}) +O(\xi) <0
\end{align*}
holds for a small enough $\xi$. This yields that
\[
E_1(N) \ll_\xi R^{-\xi} +R^{ (\{\gamma\}-2+2\xi)/(2^{d+2}-2) }.
\]
 Therefore if $\xi$ is sufficiently small and $R$ is sufficiently large, then the following holds:   
\begin{align*}
\frac{1}{16abc^2} +O_{\xi}\left(E_1(N)\right)\geq \frac{1}{32abc^2}.
\end{align*}
Hence, in this case, $\#(A \cap (N,2N]) \geq \#(A_1\cap (N,2N])\geq N/(32abc^2)>0$, which implies that there exists $n_0\in A$ satisfying \eqref{Inequality-n_0}.\\

\textbf{Step 2.} We next discuss the case when \eqref{Inequality-assumption_on_L} is false, that is to say, there exist $h_1,h_2\in \mathbb{Z}$ with $0<\max\{|h_1|,|h_2|\}\leq N^{\xi}$ such that
\begin{equation}\label{Inequality-assumption_on_L_2}
|L(h_1,h_2)|<N^{-\psi}. 
\end{equation}
We observe that $h_1$ has the opposite sign of $h_2$, since if not, $1/c \leq  |L(h_1,h_2)| <N^{-\psi} $ holds, which causes a contradiction when $R$ is sufficiently large. Thus we may assume that $h_1<0<h_2$ by multiplying the both sides of \eqref{Inequality-assumption_on_L_2} by $|(-1)|$ if necessary. Let $h_1'=-h_1$, and $\theta=L(h_1,h_2)/h_2$. 

In the case $\theta\geq 0$, by letting 
\begin{equation}
A_2=\left\{n\in [1,N^{\psi/\alpha}/(8bc)^{1/\alpha}]\cap \mathbb{N} \colon 0\leq \{(nX)^\alpha/(ch_2) \}<\frac{1}{8abcN^\xi} \right\}, 
\end{equation}
$A_2\subseteq A$ holds. In fact, suppose $n\in A_2$. Then 
$(nX)^\alpha/c= h_2 \lfloor (nX)^\alpha/(ch_2)\rfloor +h_2 \{(nX)^\alpha/(ch_2) \}$, of which the first term is an integer and the second term belongs to $[0,1)$. This yields that $\{(nX)^\alpha/c\}= h_2 \{(nX)^\alpha/(ch_2) \} $. Thus we obtain $0\leq \{(nX)^\alpha/c\}<1/(4ac)$. Further, since
\begin{gather*}
(nY)^\alpha/c=\frac{h'_1}{ch_2}(nX)^\alpha +n^\alpha \theta=h'_1 \lfloor (nX)^\alpha/(ch_2)\rfloor+h'_1\{(nX)^\alpha/(ch_2) \}+n^\alpha \theta,\\
h'_1 \lfloor (nX)^\alpha/(ch_2)\rfloor\in \mathbb{Z}, \quad 0\leq h'_1\{(nX)^\alpha/(ch_2) \}+n^\alpha \theta<\frac{1}{8bc}+\frac{1}{8bc}=\frac{1}{4bc},
\end{gather*}
we have $\{(nY)^\alpha/c\}= h'_1\{(nX)^\alpha/(ch_2) \}+n^\alpha \theta$ and $0\leq \{(nY)^\alpha/c \}< 1/(4bc)$. Hence, we obtain $A_2 \subseteq A_1 \subseteq A$. 

We next evaluate the distribution of $A_2$. Let $V=N^{\psi/\alpha}/(2(8bc)^{1/\alpha})$, and $D_2(N)$ be the discrepancy of the sequence $((nX)^\alpha/(ch_2))_{V<n\leq 2V}$. Then by the inequality \eqref{Inequality-Koksma_Szusz}, 
\begin{align*}
D_2(N) &\ll \frac{1}{N^{2\xi}}+\sum_{0< |h|\leq N^{2\xi}} \frac{1}{|h|} \left|\frac{1}{V} \sum_{V<n\leq 2V} e((h/(ch_2))X^\alpha n^\alpha) \right|.
\end{align*}
From Lemma~\ref{Lemma-van_der_Corput} with $k=d+2$ and $K=\lfloor N^{2\xi} \rfloor$, the following holds:
\begin{align*}
D_2(N)&\ll  \frac{1}{N^{2\xi}}+\sum_{0<|h|\leq N^{2\xi}} \frac{1}{|h|} 
\left( \left(\frac{|h|X^\alpha}{ch_2} V^{\alpha-d-2} \right)^{1/(2^{d+2}-2)} +\frac{ \left(\frac{|h|X^\alpha}{ch_2} V^{\alpha-d-2} \right)^{-1/(2^{d+2}-2)} }{ V^{1/2^d}}\right)\\
&\ll \frac{1}{N^{2\xi}}+\left(X^\gamma N^{2\xi}V^{\{\gamma\}-2}\right)^{1/(2^{d+2}-2)}
+N^\xi V^{(-1+2^{-d})/(2^{d+1}-1)}\\
&\ll \frac{1}{N^{2\xi}}+\left(X^\gamma N^{2\xi+\psi(\{\gamma\}-2)/\gamma}\right)^{1/(2^{d+2}-2)}
+N^{\xi+\psi(-1+2^{-d})/(\gamma(2^{d+1}-1))}.
\end{align*}
Let $E_2(N)$ be the most right-hand side. Now by \eqref{Equation-N}, we have
\begin{equation}\label{Inequality-(X+Y)2nd}
 X^\gamma N^{2\xi +\psi (\{\gamma\}-2  )/\gamma }\ll R^{2\xi+\psi (\{\gamma\}-2)/\gamma } (X+Y)^{e'}.
\end{equation}
The exponent $e'$ of $(X+Y)$ on the right-hand side of \eqref{Inequality-(X+Y)2nd} is equal to
\begin{align*}
e'&= \gamma+ \left(2\xi +\frac{\psi}{\gamma}(\{\gamma\}-2) \right)\left(\frac{\gamma^2}{(2+\{\beta\}-2^{1-\lfloor \beta\rfloor}) (2-\{\gamma\})}+\epsilon \right)\\
& =\gamma - \gamma\cdot \frac{2+\{\beta\}-2^{1-\lfloor \beta\rfloor}+O(\xi)}{2+\{\beta\}-2^{1-\lfloor \beta\rfloor}}-\epsilon\cdot \frac{\psi}{\gamma} (2-\{\gamma\})+O(\xi) \\
&=-\epsilon\cdot \frac{\psi}{\gamma} (2-\{\gamma\})+O(\xi).
\end{align*}
Here we use \eqref{Equation-psi}.
This yields that for a small enough $\xi$,
\[
	E_2(N) \ll N^{-2\xi}+(R^{2\xi+\psi(\{\gamma\}-2)/\gamma }(X+Y)^{e'})^{1/(2^{d+2}-2)} + N^{\xi +\psi (-1+2^{-d})/(\gamma (2^{d+1}-1)) }\ll N^{-2\xi}.
\]
Therefore, if necessary, by making $\xi$ smaller and $R$ larger, we get 
\[
\frac{\#(A_2\cap (V,2V])}{V}=\frac{1}{8abcN^\xi} +O( E_2(N))\geq \frac{1}{16abcN^\xi}>0.
\]
Hence by $\psi<\alpha$, there exists $n_0\in A$ such that  
\[
 n_0\ll_{\epsilon}   ((X+Y)^{\psi/\alpha})^{ \gamma^2/((2+\{\beta\}-2^{1-\lfloor \beta\rfloor} )(2-\{\gamma\})) +\epsilon }\leq    (X+Y)^{\gamma^2/((2+\{\beta\}-2^{1-\lfloor \beta\rfloor} )(2-\{\gamma\})) +\epsilon  },
\]
which is the inequality \eqref{Inequality-n_0}. In the case $\theta<0$, let $\theta'= L(h_1,h_2)/h_1 >0$.  By switching the roles of $(\theta, X^\alpha)$ and $(\theta', Y^{\alpha})$ and repeating a similar discussion to the case $\theta\geq 0$, we also find $n_0\in A$ satisfying \eqref{Inequality-n_0}.
\end{proof}

\begin{lemma}\label{Lemma-eta}
For all $\alpha>0$ and $X,Y,Z\in \mathbb{N}$, define
\[
\eta(\alpha, X,Y,Z)= \min\left\{ \frac{\log{\left( (\lfloor W ^\alpha \rfloor+1 )W^{-\alpha} \right)} }{\log W}\colon W=X,Y,Z \right\}.
\]
For all $\alpha>0$ and $X,Y,Z\in \mathbb{N}$, if $a\lfloor X^\alpha \rfloor +b\lfloor Y^\alpha \rfloor=c\lfloor Z^\alpha \rfloor$ holds, then for all $\tau\in (\alpha, \alpha+\eta(\alpha,X,Y,Z))$, we have
\[
a\lfloor X^\tau \rfloor +b\lfloor Y^\tau \rfloor=c\lfloor Z^\tau \rfloor.
\]
\end{lemma}

\begin{proof}

The claim is clear since we observe that
\begin{gather*}
\lfloor X^\alpha\rfloor=  \lfloor X^\tau\rfloor,\quad \lfloor Y^\alpha\rfloor=  \lfloor Y^\tau\rfloor,\quad  \lfloor Z^\alpha \rfloor=  \lfloor Z^\tau\rfloor
\end{gather*}
for all  $\tau\in (\alpha, \alpha+\eta(\alpha,X,Y,Z))$.
\end{proof}

\section{Lemmas II}\label{Section-dimension}
Let $x_0>0$ be a large parameter. For each $x\geq x_0$, let $K(x)\subseteq \mathbb{N}$ be a non-empty finite set. For each $x\geq x_0$ and $z\in K(x)$, let $\theta (x,z)$ and $\ell(x,z)$ be positive real numbers, and define an interval $I(x,z)=(\theta(x,z),\ \theta(x,z)+\ell(x,z))$. For each $x\geq x_0$, define
\[
G_x=\bigcup_{z\in K(x)} I(x,z).
\]
Let us consider the following conditions:
\begin{itemize}
\item[(C1)] for all integer $x\geq x_0$, $K(x)$ does not contain the multiples of $x$;
\item[(C2)] for all integers $x\geq x_0$ and $z\in K(x)$, the gap between $z$ and the minimum $z'\in K(x)$ such that $z<z'$ is at most $2$;
\item[(C3)] there exists $Q_1>0$ such that for all $x\geq x_0$,
\[
\max \left(\inf \{|\beta-\alpha| \colon \alpha\in G_x\},\ \inf\{|\gamma+x^{-2}-\alpha| \colon  \alpha \in  G_x\} \right)\leq Q_1x^{-1};
\]
\item[(C4)] there exists a real number $\kappa \in (0,\infty) \setminus\{1\}$ such that for all $x\geq x_0$ and $z\in K(x)$,
\[
\theta(x,z)=\frac{\log \kappa}{\log(z/x)}+O \left(\frac{1}{x^2\log x}\right);  
\]
\item[(C5)] there exist $Q_2,Q_3>0$ and $q>2$ such that for all $x\geq x_0$ and $z\in K(x)$,
\[
 Q_2 x^{-q} \leq  \ell(x,z) \leq Q_3x^{-\beta};
\]
\item[(C6)] for all integer $x\geq x_0$, $G_x\subseteq (\beta,\gamma+x^{-2})$ holds;
\item[(C7)] for all integers $x\geq x_0$ and $z\in K(x)$, there exists a pairwise distinct tuple $(X(x,z), Y(x,z), Z(x,z)) \in \mathbb{N}^3$ such that for all $\tau \in I(x,z)$,
\[
a\lfloor X(x,z)^\tau \rfloor +b\lfloor Y(x,z)^\tau \rfloor =c\lfloor Z(x,z)^\tau \rfloor,\quad X(x,z)\geq x.
\]
\end{itemize}

\begin{proposition} \label{Lemma-hausdorff_dimension} Suppose that there exist $x_0$, $K(x)$, $\theta(x,z)$, and $\ell(x,z)$ satisfying $(\mathrm{C}1)$ to $(\mathrm{C}7)$. Then we have
\[
\dim_\mathrm{H}(\{\alpha \in [\beta,\gamma] \colon \text{$ax+by=cz$ is solvable in $\PS(\alpha)$}  \}) \geq \frac{2}{q}.  
\]
\end{proposition}

\begin{remark} The idea of the proof of Proposition~\ref{Lemma-hausdorff_dimension} comes from the proof of Jarn\'{i}k's theorem in Falconer's book \cite[Theorem~10.3]{Falconer}. Jarn\'{i}k's theorem states that for every $q>2$ the set of $\alpha\in [0,1] $ such that there exist infinitely many $x,z \in \mathbb{N}$ satisfying $|\alpha-z/x|\leq x^{-q}$ has Hausdorff dimension $2/q$.
\end{remark}

The goal of this section is to prove Proposition~\ref{Lemma-hausdorff_dimension}. Suppose that there exist  $x_0$, $K(x)$, $\theta(x,z)$, and $\ell(x,z)$ satisfying the conditions (C1) to (C7), and choose such $x_0$, $K(x)$, $\theta(x,z)$, and $\ell(x,z)$. Take constants $Q_1$, $Q_2$, $Q_3$, $\kappa$, $q$ which appear in the conditions (C3) to (C5). Let $x_1>0$ and $U_1>0$ be large parameters depending on $a$, $b$, $c$, $d$, $\beta$, $\gamma$, $Q_1$, $Q_2$, $Q_3$, $ \kappa$, $x_0$, $q$. We do not indicate the dependence of those parameters, hereinafter. Let $p$ denote a variable running over prime numbers.

\begin{lemma}\label{Lemma-distance_of_intervals}
There exists $B_1>0$ such that for all $p\geq x_1$ and distinct $z,z'\in K(p)$, two intervals $I(p,z)$ and $I(p,z')$ are separated by a gap of at least  $B_1p^{-1}$ if $x_1$ is sufficiently large. 
\end{lemma}

\begin{proof} 
By the conditions (C4) and (C6), for all $p\geq x_1$ and $z\in K(p)$, we have 
\begin{equation}\label{Inequality-beta_kappa_gamma}
\frac{\beta}{2} \leq \frac{\log \kappa}{\log (z/p)} \leq 2\gamma
\end{equation}
if $x_1$ is sufficiently large. This implies that 
\begin{equation}\label{Inequality-pzp}
p\ll z\ll p.
\end{equation}
By the condition (C4) and the inequalities \eqref{Inequality-beta_kappa_gamma} and \eqref{Inequality-pzp}, there exists $B_0>0$ such that
\begin{align*}
|\theta(p,z)-\theta(p,z')|&=\left|\frac{\log \kappa}{\log \frac{z}{p}}- \frac{\log \kappa}{ \log \frac{z'}{p}}+O\left(\frac{1}{p^{2}\log p}\right) \right|\geq \frac{|\log \kappa||\log \frac{z'}{z}| }{|\log\frac{z}{p}||\log  \frac{z'}{p}|  } + O\left(\frac{1}{p^{2}\log p}\right)\\
&\geq \frac{\beta^2 }{4|\log \kappa|} \log\left(\frac{z+1}{z}\right)  + O\left(\frac{1}{p^{2}\log p}\right)
\geq B_0 p^{-1}
\end{align*} 
for all $p\geq x_1$ and all $z,z'\in K(p)$ with $z< z'$. Further, since  $\ell (p,z) \leq Q_3 p^{-2}$ holds by (C5), there exists $B_1>0$ such that for all $p\geq x_1$ and distinct $z,z'\in K(p)$, two intervals $I(p,z)$ and $I(p,z')$ are separated by a gap of at least  
\begin{equation}\label{Inequality-distance_of_intervals}
B_0p^{-1} - Q_3 p^{-2} \geq B_{1} p^{-1}
\end{equation}
if $x_1$ is sufficiently large. 
\end{proof}

Now we call the open interval $I(p,z)$ $(z\in K(p))$ a basic interval of $G_p$ for all $p \geq x_1$. For each $U\geq U_1$, define
\[
H_U=\bigcup_{\substack{U<p\leq 2U \\ p \text{: prime}} } G_p. 
\]
For all $U<p\leq 2U$, we also call the basic interval of $G_p$ a basic interval of $H_U$. 

\begin{lemma}\label{Lemma-distance_of_intervals_of_H_U}
There exist $B_2,B_3>0$ such that for any $U\geq U_1$, all distinct basic intervals of $H_U$ are separated by gaps of at least $B_2U^{-2}$, and  the length of each basic interval of $H_U$ is at least $B_3U^{-q}$ if $U_1$ is sufficiently large. 
\end{lemma}

\begin{proof}We take distinct prime numbers $p$ and $p'$ with $U<p, p'\leq 2U$. Then, for all $z\in K(p)$ and $z'\in K(p')$, the condition (C4), the inequality \eqref{Inequality-beta_kappa_gamma}, and the mean value theorem imply that
\begin{align*}
|\theta(p,z)-\theta(p',z')|&\geq \left|\frac{\log \kappa}{\log (z/p)} - \frac{\log \kappa}{\log (z'/p')}\right|+O\left(\frac{1}{U^2\log U}\right) \\
& \geq \frac{\beta ^2}{4|\log \kappa|} \left|\frac{z}{p}-\frac{z'}{p'} \right| \min\left\{\frac{p}{z},\frac{p'}{z'} \right\} +O\left(\frac{1}{U^2\log U}\right).
\end{align*}
We may assume that $p'/z'>p/z$. By the condition (C1), $z$ and $p$ are coprime, which yields that $|zp'-z'p|\geq 1$. Therefore we obtain 
\[
\left|\frac{z}{p}-\frac{z'}{p'} \right|  \min\left\{\frac{p}{z},\frac{p'}{z'} \right\} =\left|\frac{z}{p}-\frac{z'}{p'} \right| \frac{p}{z} \geq \frac{1}{p'z}\gg U^{-2}
\]
by the inequalities \eqref{Inequality-pzp} and $U<p, p'\leq 2U$. Therefore for all $U\geq U_1$, we have
\begin{equation}\label{Inequality-distance_of_intervals_2}
|\theta(p,z)-\theta(p',z')|\gg \frac{1}{U^2}
\end{equation}
if $U_1$ is sufficiently large. Further, for all $U<p\leq 2U$ and $z\in K(p)$, we have $\ell (p,z)\ll  U^{-\beta}$. Hence there exists $D_1>0$ such that for all distinct prime numbers $U<p,p' \leq 2U$, $z\in K(p)$, and $z'\in K(p')$, the intervals $I(p,z)$ and $I(p',z')$ are separated by gaps of at least $D_1U^{-2}$.  By combining with Lemma~\ref{Lemma-distance_of_intervals}, there exists $D_2>0$ such that distinct basic intervals of $H_U$ are separated by gaps of at least $D_2U^{-2}$. Furthermore by (C5), for all $U<p\leq 2U$ and $z\in K(p)$, we have $Q_2\cdot 2^{-q} U^{-q}\leq \ell(p,z)$. In conclusion, we find that all distinct basic intervals of $H_U$ are separated by gaps of at least $B_2U^{-2}$, and has length of at least $B_3U^{-q}$, where we let $B_2=D_2$ and $B_3=Q_2\cdot 2^{-q}$.
\end{proof}

\begin{lemma}\label{Lemma-the_number_of_intervals}
There exists $B_4>0$ such that the following statement holds: for every $U\geq U_1$, if an open interval $I \subset (\beta, \gamma+p^{-2})$ satisfies
\begin{equation}\label{Inequality-I}
3B_4/\diam(I) <U<p\leq 2U,
\end{equation}
then the open interval $I$ completely includes at least
\begin{equation}\label{Inequality-the_number_of_intervals}
\frac{U^2}{6B_4 \log U}\cdot \diam(I)
\end{equation}
basic intervals of $H_U$. 
\end{lemma}
\begin{proof}
By (C4), \eqref{Inequality-beta_kappa_gamma}, and \eqref{Inequality-pzp},  there exists $D_3>0$ such that for every $z\in K(p)$ and the minimum $z'\in K(p)$ with $z'>z$,
\begin{align}\nonumber
|\theta(p,z)-\theta(p,z')|&=\left|\frac{\log \kappa}{\log (z/p)}- \frac{\log \kappa}{ \log (z'/p)}+O\left(\frac{1}{p^{2}\log p}\right) \right|\\ \label{Inequality-D_3}
&\leq \frac{4\gamma^2}{|\log \kappa|}\cdot \frac{1}{z}\cdot |z-z'| + O\left(\frac{1}{p^{2}\log p}\right)
\leq D_3 p^{-1}.
\end{align}
Here we apply (C2) when we deduce the last inequality. From (C3) ,(C6) and \eqref{Inequality-D_3}, there exists $B_4>0$ such that
\begin{align*}
(\beta, \gamma+p^{-2}) &\subseteq \left(\beta, \beta+B_4p^{-1} \right) \cup \left(\bigcup_{z\in K(p)} \left(\theta(p,z), \theta(p,z)+B_4p^{-1} \right) \right)\\
&\hspace{10pt}\cup   \left(\gamma+p^{-2}-B_4p^{-1}, \gamma+p^{-2}\right).
\end{align*}
Therefore for all $U\geq U_1$ and $U<p\leq 2U$, any open interval $I \subset (\beta, \gamma+p^{-2})$ with \eqref{Inequality-I} completely includes at least $B_4^{-1}p\cdot \diam(I)-2\geq (3B_4)^{-1}U\cdot \diam(I) $ basic intervals of $G_p$. Hence, by the prime number theorem, the open interval $I$ completely includes at least \eqref{Inequality-the_number_of_intervals} basic intervals of $H_U$ for a large enough $U_1$.
\end{proof}

\begin{proof}[Proof of Proposition~\ref{Lemma-hausdorff_dimension}]  Let $B_3$ and $B_4$ be constants as in Lemma~\ref{Lemma-distance_of_intervals_of_H_U} and Lemma~\ref{Lemma-the_number_of_intervals}, respectively. Let $u_1=U_1$. For every $k=2,3, \ldots$, we put
\[
u_k=\max \{u_{k-1}^k,\ \lceil3(B_4/B_3)u_{k-1}^{q} \rceil,\ u_{k-1}+1 \}, 
\]
and $B_5=B_3/(6B_4)$. Let $E_1$ be the open interval $(\beta,2\gamma)$. For every $k=2,3,\ldots$, let $E_k$ be the union of basic intervals of $H_{u_k}$ which are completely included by $E_{k-1}$. Let $F$ be the intersection of all $E_k$'s.  
Define $m_1=1$, and for $k\geq 2$, define 
\[
m_k=\frac{u_{k}^2}{6B_4\log u_{k} }  B_3 u_{k-1}^{-q}=B_5\frac{u_k^{2}u_{k-1}^{-q}}{\log u_k}. 
\]
Then each $(k-1)$-st level interval of $F$ includes at least $m_k$ $k$-th level intervals. In fact, it follows by applying Lemma~\ref{Lemma-the_number_of_intervals} and Lemma~\ref{Lemma-distance_of_intervals_of_H_U} that each $(k-1)$-st level interval of $F$ has length at least $B_3 u_{k-1}^{-q}$. In addition, by Lemma~\ref{Lemma-distance_of_intervals_of_H_U}, disjoint $k$-th level intervals of $F$ are separated by gaps of at least $\delta_k=B_2u_k^{-2}$. Therefore, Lemma~\ref{Lemma-lower_bounds_for_dimension} implies that 
\begin{align*}
\dim_{\mathrm{H}} F& \geq \liminf_{k\rightarrow \infty}\frac{\log\left(m_1m_2\cdots m_{k-1} \right) }{-\log(\delta_{k}m_{k}) }\\
&=\liminf_{k\rightarrow \infty} \frac{2\log u_{k-1} + \log\left(B_5^{k-2} u_1^{-q} (u_2\cdots u_{k-2})^{2-q}(\log u_2)^{-1}\cdots (\log u_{k-1})^{-1}  \right) }{q\log u_{k-1}+\log k+\log\log u_{k-1}-\log (B_2B_5) }.
\end{align*}
Since we have $u_{k+1}\geq u_{k}$ for all $k\geq 1$ and $ u_{k}=u_{k-1}^k$ for a large enough $k\geq 1$, we have 
\[
\left|\log\left(B_5^{k-2} u_1^{-q} (u_2\cdots u_{k-2})^{2-q}(\log u_2)^{-1}\cdots (\log u_{k-1})^{-1}  \right)\right|
\ \ll \log u_{k-2}.
\]
Therefore, since $\log u_{k-2}/ \log u_{k-1}= 1/(k-1) \rightarrow 0$ as $k\rightarrow \infty$, we get
\[
\dim_{\mathrm{H}} \left( \bigcap_{k=1}^\infty E_k \right)\geq \frac{2}{q}.
\]

We finally show that for any $\tau \in F$, the equation $ax+by=cz$ is solvable in $\PS(\tau)$ and $\tau\in [\beta,\gamma]$. If this claim is true, we get the conclusion of Proposition~\ref{Lemma-hausdorff_dimension} by the monotonicity of the Hausdorff dimension. 

Take any $\tau \in F$. It is clear that $\tau \in [\beta,\gamma]$ since the condition (C6) yields $H_{u_k}\subseteq (\beta, \gamma+u_k^{-2})$, which implies $F \subseteq [\beta,\gamma]$. Further, by (C7), for all $k > 1$, there exist a prime number $u_k<p_k\leq 2u_k$ and $z_k\in K(p_k)$ such that we find a pairwise distinct tuple $(X(p_k,z_k), Y(p_k,z_k),Z(p_k,z_k) ) \in \mathbb{N}^3$ such that  
\[
a\lfloor X(p_k,z_k)^\tau \rfloor+b\lfloor Y(p_k,z_k)^\tau \rfloor=c\lfloor Z(p_k,z_k)^\tau \rfloor, \quad  X(p_k,z_k) \geq p_k.
\]
Since $ X(p_k,z_k) \geq p_k\geq u_k\rightarrow \infty$ as $k\rightarrow \infty$,  the equation $ax+by=cz$ is solvable in $\PS(\tau)$.
\end{proof}

\section{Proof of Theorem~\ref{Theorem-fractal_dimension}} \label{Section-proof_of_main_theorem}

Fix any $a,b,c\in \mathbb{N}$. Let $\epsilon>0$ be an arbitrarily small real number. Let $d=\lfloor s \rfloor $ and choose real numbers $\beta,\gamma$ with $d\leq s< \beta<\gamma<\min\{t ,d+1\} $. By the monotonicity of the Hausdorff dimension, we have
\begin{align}\nonumber
&\dim_\mathrm{H}(\{\alpha \in [s,t] \colon \text{$ax+by=cz$ is solvable in $\PS(\alpha)$}  \}) \\ \label{Inequality-mono_hausdorff}
&\geq  \dim_\mathrm{H}(\{\alpha \in [\beta,\gamma] \colon \text{$ax+by=cz$ is solvable in $\PS(\alpha)$}  \}).
\end{align}

Take $\alpha(x,z)$ as in Lemma~\ref{Lemma-alpha(x,z)_1_2} and Lemma~\ref{Lemma-alpha(x,z)}. Let $K(x)=J_{a,b,c}(x)$, $\theta(x,z)=\alpha(x,z)$. We give $\ell(x,z)$ later. Let us check the conditions (C1) to (C7), and apply Proposition~\ref{Lemma-hausdorff_dimension}. \\

\textbf{Step~1.} In the case $a>c$, by Lemma~\ref{Lemma-alpha(x,z)_1_2}, for all $x\geq x_0$ and $z\in J_{a,b,c}(x)$, 
\[
ax^{\alpha(x,z)}+b=cz^{\alpha(x,z)}
\] 
holds. Thus by Lemma~\ref{Lemma-n_0}, there exists $n_0\in \mathbb{N}$ such that
\begin{gather}\label{Equation-xyzn_0}
a\lfloor (n_0x)^\alpha \rfloor +b\lfloor n_0^{\alpha} \rfloor =c\lfloor(n_0z)^\alpha \rfloor,\\ \label{x,y,z<1/2}
 \max(\{(n_0x)^\alpha\}, \{(n_0)^\alpha\}, \{(n_0z)^\alpha\})<1/2,\\ \label{Inequality-n_0_a>c}
 n_0 \ll_{\epsilon}  x^{\gamma^2/((2+\{\beta\}-2^{1-\lfloor \beta\rfloor} )(2-\{\gamma\})) +\epsilon  }. 
\end{gather}
Define $\eta$ as in Lemma~\ref{Lemma-eta}. Let $\ell(x,z)=\eta(\alpha(x,z), n_0x, n_0, n_0z)$. The condition (C1) is clear from the definition of $J_{a,b,c}(x)$. The condition (C2) is also clear since we find at most one multiple of $x$ in any $3$-consective integers if $x_0\geq 3$. Lemma~\ref{Lemma-alpha(x,z)_1_2} implies (C4). By Lemma~\ref{Lemma-eta}, for each $x\geq x_0$ and $z\in J_{a,b,c}(x)$, each $\tau\in (\alpha(x,z),\alpha(x,z) +\ell(x,z))$ satisfies
\[
a\lfloor (n_0x)^\tau \rfloor +b\lfloor n_0^{\tau} \rfloor =c\lfloor(n_0z)^\tau \rfloor, \quad n_0x\geq x.
\]
Therefore we have (C7). Let us show (C3), (C5), (C6).

We show (C3). Let $x$ be an integer with $x\geq x_0$. For each $i\in \{1,2\}$, let
\[
z_{1,i}=\left\lfloor \left(\frac{b}{cx^2\log x}+\frac{a}{c}\right)^{1/\gamma}  x \right\rfloor+i,\quad 
z_{2,i}=\lfloor (a/c)^{1/\beta} x\rfloor -i.
\]
Remark that $J_{a,b,c}(x)$ does not contain multiples of $x$. Thus we do not know whether $z_{1,i} ,z_{2,i}\in J_{a,b,c}(x)$ for each $i\in \{1,2\}$. However, by (C2), there exist $i_1,i_2\in \{1,2\}$ such that $z_{1,i_1}, z_{2,i_2}\in J_{a,b,c}(x)$.  Lemma~\ref{Lemma-alpha(x,z)_1_2} implies that
\[
 \alpha(x,z_{1,i_1}) = \frac{\log (a/c)}{\log(z_{1,i_1}/x) } +O\left(\frac{1}{x^2\log x}\right).
\]
Here we have
\begin{align*}
\log(z_{1,i_1}/x)&=\log\left(  \left(\frac{b}{cx^2\log x}+\frac{a}{c}\right)^{1/\gamma} +O(x^{-1})\right)\\
&=\frac{1}{\gamma} \log(a/c) +\log\left(1+O\left(\frac{b}{a\gamma x^2\log x}\right)+O(x^{-1}) \right)\\
&=\frac{1}{\gamma} \log(a/c) +O(x^{-1}).
\end{align*}
Therefore
\[
 \alpha(x,z_{1,i_1}) = \frac{\log (a/c)}{ \frac{1}{\gamma} \log(a/c) +O(x^{-1}) } +O\left(\frac{1}{x^2\log x}\right)= \gamma+O(x^{-1}).
\]
Similarly, we have $\alpha(x,z_{2,i_2})=\beta+O( x^{-1})$. Hence we obtain (C3).

We next show (C5). For all $x\geq x_0$ and $z\in J_{a,b,c}(x)$, $x< z$ holds by the definition of $J_{a,b,c}(x)$. Hence we obtain
\[
\ell(x,z)=\eta(\alpha(x,z), n_0x, n_0, n_0z)=\frac{\log{\left( (\lfloor W^\alpha \rfloor+1 )W^{-\alpha} \right)}}{\log W},
\]
where $W$ is one of $n_0x, n_0$, or $n_0z$. By $\beta<\alpha(x,z)$, we have $\ell(x,z)\leq \log (1+(n_0 x)^{-\beta}) \leq x^{-\beta}
$.  Further, by the facts \eqref{x,y,z<1/2}, \eqref{Inequality-n_0_a>c}, and $1 < x < z$, we have
\[
\ell(x,z)\geq \frac{\log(1+2^{-1} W^{-\alpha}) } {\log W}\gg \frac{1}{ (n_0z)^\gamma \log (n_0z) }\gg_{\epsilon} x^{-q},
\]
where let 
\[
q=q(\beta, \gamma,\epsilon)=(\gamma+\epsilon)\left( \frac{\gamma^2}{(2+\{\beta\}-2^{1-\lfloor \beta\rfloor} )(2-\{\gamma\})} +1 + \epsilon\right). 
\]
Therefore (C5) holds. The remaining condition (C6) is clear since $\beta <\alpha(x,z)<\gamma$ and $\alpha(x,z)+\ell(x,z)<\gamma+x^{-2}$ by (C5).\\

\textbf{Step~2.} In the case $c>a$, define $n_0$, $\ell(x,z)$, $q(\beta, \gamma,\epsilon)$ by the same way in Step~1. The condition (C1) is clear since $z<x$ by the definition of $J_{a,b,c}(x)$. The condition (C2) is also clear since $J_{a,b,c}(x)$ forms a set of consecutive integers. Lemma~\ref{Lemma-alpha(x,z)_1_2} implies (C4). Similarly to the discussion in Step~1, we have (C5), (C6), and (C7). 
Let us show the remaining condition (C3). Let $x$ be an integer with $x\geq x_0$. Let
\[
z_{1}=\left\lfloor \left(\frac{a}{c-b(x^2\log x)^{-1}}\right)^{1/\beta}  x \right\rfloor+1,\quad 
z_{2}=\lfloor (a/c)^{1/\gamma} x\rfloor -1.
\]
We observe that $z_1,z_2\in J_{a,b,c}(x)$ if $x_0$ is sufficiently large. Lemma~\ref{Lemma-alpha(x,z)_1_2} implies that
$\alpha(x,z_1) = \beta+O \left( x^{-1} \right)$ and $\alpha(x,z_2) =\gamma+O \left( x^{-1} \right)$. Therefore we have (C3).\\

\textbf{Step~3.} In the case $a=b=c$, without loss of generality, we may assume that $a=b=c=1$. By Lemma~\ref{Lemma-alpha(x,z)}, for all $x\geq x_0$ and $z\in J_{1,1,1}(x)$, by letting $X=X(x,z)=x^2 \lceil \log x \rceil$, $Y=Y(x,z)=x^2\lceil \log x\rceil+1$, $Z=Z(x,z)=zx\lceil \log x \rceil$, we have
\[
 X^{\alpha(x,z)}+Y^{\alpha(x,z)}=Z^{\alpha(x,z)}.
\]
Therefore, from Lemma~\ref{Lemma-n_0}, there exists $n_0\in \mathbb{N}$ such that
\begin{gather} \nonumber
\lfloor (n_0X)^\alpha \rfloor +\lfloor (n_0Y )^{\alpha} \rfloor =\lfloor(n_0Z)^\alpha \rfloor, \\ \label{X,Y,Z<1/2}
\max(\{(n_0X)^\alpha\}, \{(n_0 Y)^\alpha\}, \{(n_0Z)^\alpha\})<1/2, \\ \nonumber
n_0 \ll_{\epsilon}  (X+Y)^{\gamma^2/((2+\{\beta\}-2^{1-\lfloor \beta\rfloor} )(2-\{\gamma\})) +\epsilon }. 
\end{gather} 
Here by letting $r=r(\gamma,\beta,\epsilon)=\gamma^2/((2+\{\beta\}-2^{1-\lfloor \beta\rfloor} )(2-\{\gamma\})) +\epsilon$, we obtain 
\begin{equation}\label{Inequality-n_0_a=b=c}
n_0 \ll_{\epsilon} x^{(2+\epsilon)r}.
\end{equation}
Let $\ell(x,z)=\eta(\alpha(x,z), n_0X, n_0Y, n_0Z)$.

The condition (C1) is clear since $x<z<2x$ by the definition of $J_{1,1,1}(x)$. The condition (C2) is also clear since $J_{1,1,1}(x)$ forms a set of consecutive integers. Lemma~\ref{Lemma-alpha(x,z)} implies (C4). By Lemma~\ref{Lemma-eta}, for all $x\geq x_0$, $z\in J_{1,1,1}(x)$, each $\tau\in (\alpha(x,z),\alpha(x,z) +\ell(x,z))$ satisfies
\[
\lfloor (n_0X)^\tau \rfloor +\lfloor (n_0Y )^{\tau} \rfloor =\lfloor(n_0Z)^\tau \rfloor,\quad n_0X\geq x.
\]
Therefore (C7) holds. Therefore it suffices to show (C3), (C5), and (C6).

Let us show (C3). Take any integer $x\geq x_0$. Let 
\[
z_{1}=\left\lfloor 2^{1/\gamma}  (x+(x\lceil \log x\rceil)^{-1}) \right\rfloor+1,\quad 
z_{2}=\lfloor 2^{1/\beta} x \rfloor -1.
\]
It follows that $z_{1}, z_{2}\in J_{1,1,1}(x)$ if $x_0$ is sufficiently large. Then Lemma~\ref{Lemma-alpha(x,z)} implies that $\alpha(x,z_{1}) =\gamma+O \left( x^{-1} \right)$ and $\alpha(x,z_2) =\beta+O \left( x^{-1} \right)$. Therefore we have (C3).

We next show (C5). Let $x$ be an integer with $x\geq x_0$ and $z\in J_{1,1,1}(x)$. They are clear that $x<z$ and $X(x,z)< Y(x,z)< Z(x,z)$. Recall that
\[
\ell(x,z)=\eta(\alpha(x,z), n_0X, n_0Y, n_0Z)=\frac{\log{\left( (\lfloor W^\alpha \rfloor+1 )W^{-\alpha} \right)}}{\log W},
\]
where $W$ is one of $n_0 X, n_0 Y$, or $n_0Z$. Therefore, by $\beta<\alpha$, we have
$
\ell(x,z)\leq \log (1+(nZ)^{-\beta})\leq Z^{-\beta} \leq x^{-\beta}.
$
Further, by the facts \eqref{X,Y,Z<1/2} and \eqref{Inequality-n_0_a=b=c}, we obtain
\[
\ell(x,z)\geq \frac{\log(1+2^{-1} W^{-\alpha}) } {\log W} \gg \frac{1}{ (n_0Z)^\gamma \log (n_0Z) }\gg_{\epsilon} x^{-(2+\epsilon)(\gamma+\epsilon)(r+1) }.
\]
Hence, (C5) holds. The condition (C6) is clear since $\beta <\alpha(x,z)<\gamma$ and $\alpha(x,z)+\ell(x,z)<\gamma+x^{-2}$ by (C5).\\

\textbf{Step~4.} By summarizing the above discussion, define
\begin{equation*}
D_{a,b,c}(\beta,\gamma,\epsilon) = 
\left\{
\begin{array}{ll}
\displaystyle \frac{2}{(2+\epsilon)(\gamma+\epsilon)(r(\beta,\gamma,\epsilon)+1)}  & \text{if $a=b=c$,}\\[10pt]
\displaystyle \frac{2}{q(\beta,\gamma,\epsilon)} & \text{otherwise}. 
\end{array}
\right.
\end{equation*}
Step~1, Step~2, Step~3 and Proposition~\ref{Lemma-hausdorff_dimension} imply that  
\[
\dim_\mathrm{H}(\{\alpha \in [\beta,\gamma] \colon \text{$ax+by=cz$ is solvable in $\PS(\alpha)$}  \}) \geq D_{a,b,c}(\beta,\gamma,\epsilon).  
\]
Therefore, by \eqref{Inequality-mono_hausdorff} and by letting $\epsilon\rightarrow +0$, $\gamma \rightarrow \beta$, $\beta\rightarrow s$, we have
\[
\dim_\mathrm{H}(\{\alpha \in [s,t] \colon \text{$ax+by=cz$ is solvable in $\PS(\alpha)$}  \}) \geq D_{a,b,c}(s,s,0).  
\]
By the definitions of $q$ and $r$, we get the conclusion of Theorem~\ref{Theorem-fractal_dimension}.

\section*{Acknowledgement}
The first author was supported by JSPS KAKENHI Grant Number JP20K14292. The second author was supported by JSPS KAKENHI Grant Number JP19J20878.

\bibliographystyle{amsalpha}
\bibliography{references_arXiv}
\end{document}